\documentclass[12pt,twoside,reqno]{amsart}
\usepackage{amsmath, amsthm, amscd, amsfonts, amssymb, graphicx, color}
\usepackage[bookmarksnumbered, colorlinks, plainpages]{hyperref}
\usepackage[utf8]{inputenc}
\usepackage{cleveref}
\crefname{section}{§}{§§}
\Crefname{section}{§}{§§}
\let \oldsection
\section
\renewcommand{\section}{\vspace{8pt plus 4pt}\oldsection}

\newcommand{\beqa}{\begin{eqnarray*}}
\newcommand{\eeqa}{\end{eqnarray*}}
\newcommand{\beqn}{\begin{eqnarray}}
\newcommand{\eeqn}{\end{eqnarray}}

\newcommand{\R}{\mathbb R}

\newcommand{\N}{\mathbb N}

\newcommand{\mcP}{\mathcal P}

\usepackage{lipsum}
\usepackage[super]{nth}
\usepackage{xcolor}
\usepackage{tikz}
\usepackage{pgfplots}
\usepackage{mathrsfs}
\usepackage{graphics}
\usepackage{graphicx} 
\usepackage{epsfig}
\usepackage{wrapfig}
\definecolor{olive}{rgb}{0.3, 0.4, .1}
\definecolor{fore}{RGB}{249,242,215}
\definecolor{back}{RGB}{51,51,51}
\definecolor{title}{RGB}{255,0,90}
\definecolor{dgreen}{rgb}{0.,0.6,0.}
\definecolor{gold}{rgb}{1.,0.84,0.}
\definecolor{JungleGreen}{cmyk}{0.99,0,0.52,0}
\definecolor{BlueGreen}{cmyk}{0.85,0,0.33,0}
\definecolor{RawSienna}{cmyk}{0,0.72,1,0.45}
\definecolor{Magenta}{cmyk}{0,1,0,0}
\newtheorem{thm}{Theorem}[section]

\newtheorem{prop}[thm]{Proposition}

\theoremstyle{definition}
\newtheorem{defn}{Definition}[section]

\theoremstyle{remark}

\numberwithin{equation}{section}
\begin{document}
\begin{center}\large{{\bf{ On Geometric properties  of  Henstock-Orlicz spaces}}} \\
%(Dedicated to $45^{th}$ Birthday of my Advisor Prof. Bipan Hazarika) 
\vspace{0.5cm}
 
\footnotesize     Hemanta Kalita$^{1}$,  Salvador S\'anchez Perales$^{2}$ and Bipan Hazarika$^{3, \ast}$ \\
%\vspace{0.5cm}
\footnotesize $^{1}$Department of Mathematics, Gauhati University,
Guwahati 781014, Assam, India\\
%$^{2}$Department of Electrical Engineering, Computer Engineering and Mathematics, Howard University, Washington DC 20059, USA\\
%$^1$Department of Mathematics and Computer Sciences, via Vanvitelli, 1  I-06123 Perugia, Italy
%$^{1}$ Department of Mathematics, Gauhati University, Guwahati 781014, Assam, India\\
Email:  hemanta30kalita@gmail.com\\ 
\footnotesize $^{2}$Universidad Tecnol\'ogica de la Mixteca, Instituto de F\'isica y Matem\'aticas, Km. 2.5 Carretera a Acatlima, 69000 Oaxaca, Mexico.\\
Email: es21254@yahoo.com.mx	\\
\footnotesize $^{3}$Department of Mathematics, Gauhati University, Guwahati 781014, Assam, India\\
%$^{2}$Department of Mathematics, Gauhati University, D, Patkai - 797103, Nagaland, India
%$^{3}$Department of Mathematics, Rajiv Gandhi University, Rono Hills, Doimukh 791112, Arunachal Pradesh, India\\
Email: bh\_rgu@yahoo.co.in; bh\_gu@gauhati.ac.in  

\end{center}
\title{}
\author{}
\thanks{$^{\ast}$The corresponding author}
\thanks{\today} 
\begin{abstract}
In this paper we extend the theory of Henstock-Orlicz spaces with respect to vector measure. We study the integral representation of operators. Lastly we study Uniformly convexity, reflexivity and the Radon-Nikodym property of the Henstock-Orlicz spaces $\mathcal{H}^\theta(\mu_\infty).$ 
{\footnotesize {\bf{AMS subject classification \textrm{(2020)}:}}} 46E30, 46B20, 46B22, 46A80.
\end{abstract}
\maketitle

\maketitle

\pagestyle{myheadings}
\markboth{\rightline {\scriptsize    Hemanta, Salvador, Bipan}}
        {\leftline{\scriptsize  }}

\maketitle
\section{Introduction and preliminaries}
In the early $19^{th}$ century, Lebesgue's Theory of integration has taken a center stage in concrete problem of analysis. It was seen as early as 1915 with the publications of de-la Vall\`ee Poussin, although it is the Banach space research of 1920's that formally gave birth to what are later called the Orlicz spaces, first proposed by Z.W. Birnbaum and W. Orlicz. Later on this space was further developed by Orlicz himself. In the monograph \cite{Ma}, Kransoselskii and  Rutickii devoted entirely to Orlicz spaces. We refer \cite{Th,etaltin,Mm,BP,Tv,bcrg} for detailed of Orlicz space. Brooks and Dinculeanu have developed a theory of vector integration for a bounded family of measures (see \cite{BD}). In \cite{Roy}, Roy and Chakraborty developed a theory of Orlicz spaces for the case of Banach space valued functions with respect to a $\sigma-$bounded family of measures. In \cite{Roy1} Roy and Chakraborty developed integral representation as an application of their previous work of \cite{Roy}. In \cite{AK,AK1}, Kaminska discussed criteria for uniform convexity of Orlicz spaces in the case of a non atomic measure as well as in the case of a purely atomic measure. It is known that,  if $f$ is  bounded  with compact support, then following are equivalent:
       \begin{enumerate}
       	\item[(a)] $f$ is Henstock–Kurzweil integrable, 
       	\item[(b)]	$f$ is Lebesgue integrable,
       	\item[(c)] 	$f$ is Lebesgue measurable.
       \end{enumerate}
       In general, every Henstock–Kurzweil integrable function is measurable, and $f$ is Lebesgue integrable if and only if both $f$ and $|f|$ are Henstock–Kurzweil integrable. This means that the Henstock–Kurzweil integral can be thought of as a ``non-absolutely convergent version of Lebesgue integral". We refer the detailed of Henstock-Kurzweil integral found in \cite{Ra,Rh,Rh1,Kurzweil,Cs,Bs,lee}.
We developed Henstock-Orlicz spaces in \cite{HK}. In this paper we developed a theory of Henstock-Orlicz spaces for the vector valued functions with respect to a $\sigma-$bounded family of measures. Detailed of vector measure can found in \cite{BHK,BD,JD,KH}. Throughout this paper, $\Sigma$ denotes a $\sigma-$algebra of subsets of an abstract set $T \neq \emptyset.~ \mathcal{P}(T)$ be the class of all subset of $T,~\Sigma \subset \mathcal{P}(T)$ be a $\sigma-$algebra, $X$ be a Banach space, and $X^{\prime}$ be its topological dual. For each $A \in \Sigma,$ the characteristic function of $A$
\begin{eqnarray*}
ch_A(t)=\left\{
\begin{array}{ll}
1 & \text{if  } t \in A, \\ 
0 & \text{if  } t \in T \setminus A.
\end{array}\right.
\end{eqnarray*}
Recalling,  vector measure is a $\sigma-$additive set function $\mu_\infty: \Sigma \to X$, the $\sigma-$additivity of $\mu_\infty$ is equivalent to the $\sigma-$additivity of the scalar-valued set functions $x^{\prime}\mu_\infty: A \to x^{\prime}(\mu_\infty(A))$ on $\Sigma$ for every $x^{\prime} \in X^{\prime}.$
\begin{defn}
\cite{BHK} The variation $|\mu_\infty|$ of $\mu_\infty$ is defined by 
\begin{eqnarray*}
|\mu_\infty|(A)= \sup\left\{\sum_{i=1}^{r}||\mu_\infty(A_i)||~: A_i \in \Sigma,~i=1,2,..,r; A_i \cap A_j = \emptyset ~{\text{for}}~i \neq j; ~\bigcup\limits_{i=1}^{r}A_i \subset A\right\}.
\end{eqnarray*}
\end{defn}
\noindent The semi variation $||\mu_\infty||$ of $\mu_\infty$ by 
\begin{eqnarray*}
||\mu_\infty||(A)= \sup\limits_{x^{\prime} \in X^{\prime},||x^{\prime}||\leq1}|x^{\prime}\mu_\infty|(A).
\end{eqnarray*}
\noindent A function $f:T \to \mathbb{R}$ is said to be $\mu_\infty$-measurable if 
\begin{eqnarray*}\label{mumeasurability}
f^{-1}(B) \cap \{ t \in T: f(t) \neq 0 \}\in 
\Sigma 
\end{eqnarray*}
for each Borel subset $B \subset \mathbb{R}$.
\begin{defn}
\cite{BHK} 
We say that
a $\Sigma$-measurable function $f:T \to \mathbb{R}$ 
is Kluv\'{a}nek-Lewis-Henstock-Kurzweil 
$\mu_\infty$-integrable, shortly $(HKL)$ $\mu_\infty$-integrable if 
the following properties hold: 
\begin{eqnarray}\label{first}
f \text{   is   } 
|x^{\prime} \mu_\infty|-\text{Henstock-Kurzweil integrable  for each }  
x^{\prime} \in X^{\prime},
\end{eqnarray} and for every $A \in \Sigma$ there is  $x_A^{(HK)} \in X$ with
\begin{eqnarray}\label{second} 
 x^{\prime}(x_A^{(HK)})=
(HK) \int_{A}f \, d |x^{\prime}\mu_\infty|
\text{    for  all  }
x^{\prime} \in X^{\prime}, 
\end{eqnarray}
where the symbol $(HK)$ in 
denote the usual  Henstock-Kurzweil integral
of a real-valued function with respect to an 
(extended) real-valued measure.
\end{defn}
 
 \begin{defn}
 \cite{Roy} A function $f: T \to X $ is $M-$measurable if there is a sequence of simple functions from $\mathcal{S}_X(\Sigma)$ converges to $f$ M-a.e.
 \end{defn} 
 $M$ denotes a $\sigma-$bounded family of positive measures defined on $\Sigma.$ This means for each $E \in \Sigma,$ there exists a pairwise disjoint collections $\{E_i\}_{i=1}^{\infty},~E_i \in \Sigma$ such that $E= \bigcup\limits_{i=1}^{\infty}E_i $ and  $\mathcal{S}_X(E)$ is a set of all $X$ valued simple functions. If $f$ is $M-$measurable, then $f$ is $\mu_\infty$-measurable for each $\mu_\infty \in M,$ we define $M(f)=\sup\limits_{\mu_\infty \in M}H\int_{T}||f||d|x^{\prime}\mu_\infty|.$ In this paper we consider all functions are $M-$ measurable. Let $\mathcal{M}_X(M)$ denote the space of all $M-$ measurable functions $f:T \to X$ for which $M(f) < \infty,$ then $\mathcal{M}_X(M)$ is complete with respect to the semi norm $M(.).$ If $M(f)< \infty,$ then $f$ is integrable in $\mathcal{M}_X(E).$ 
%\textcolor{blue}{
	\begin{defn}\cite{Roy1}
Let $\chi:[0,\infty] \to [0, \infty]$ be a non-decreasing left continuous function such that $\chi(0)=0.$ The inverse function $\xi$ of $\chi$ is defined by $\xi(0)=0$ and $\xi(v)= \sup\{u:~\chi(u)<v\}.$ If $\lim\limits_{u \to \infty}\chi(u)=A<\infty$ then $\xi(v)=\infty$ for all $v> A.$ The functions defined below 
\begin{eqnarray*}
\theta(u) = \int_{0}^{u}\chi(t)dt\mbox{~and~} \phi(v)= \int_{0}^{v}\xi(t)dt
\end{eqnarray*}
are called conjugate Young's function and they satisfy the Young's inequality
\begin{eqnarray*}
uv \leq \theta(u)+\phi(v),~u,v \geq 0.
\end{eqnarray*}
\end{defn}
\section{Henstock-Orlicz space and vector measure}
If $f \in \mathcal{H}^\theta(|x^{\prime}\mu_\infty|),$ then there is a $k>0$ such that $\theta(\frac{f}{k}) \in H(|x^{\prime}\mu_\infty|)=L^1(|x^\prime \mu_\infty|),$ where $H(|x^\prime \mu_\infty|) $ and $L^1(|x^\prime \mu_\infty|)$ are Henstock integrable space and Lebesgue integrable space. The space $\mathcal{H}^\theta(|x^{\prime}\mu_\infty|)$ endowed with the Lexemburg norm
\begin{eqnarray}\label{hk}
||f||_{\mathcal{H}^\theta}= \inf\left\{k>0~: ~H\int_{T}\theta\left(\frac{f}{k}\right)d|x^{\prime}\mu_\infty| \leq 1 \right\}
\end{eqnarray}
is a Banach space. 
\begin{defn}
For an $M-$measurable function $f: T \to X ,$ let us define
\begin{eqnarray}\label{hk1}
||f||_{\theta, M}= \sup\limits_{\mu_\infty \in M} \sup\limits_{g \in \mathcal{S}(\Sigma)}\left\{H\int_{T}||fg||d|x^{\prime}\mu_\infty|~: M_\phi(g) \leq 1\right\}.
\end{eqnarray}
\end{defn}
\noindent Clearly (\ref{hk}) and (\ref{hk1}) are equivalent.
\begin{defn}
\cite{Roy1} If $X$ and $Y$ are two Banach spaces and $\mu_\infty:\Sigma \to L(X,Y)$ a countable additive measure, then $(\theta, M)-$variation of $\mu_\infty$ is defined by 
\begin{eqnarray*}
|\mu_\infty|_{\theta, M}(E)= \sup\sum_{i=1}^{n}||\mu_\infty(E_i)x_i||~{\text{for each~}} E \in \Sigma.
\end{eqnarray*}
\end{defn}
Here the supremum is taken over all simple functions $f=\sum\limits_{i=1}^{n}x_ich(E_i)$ belonging to $\mathcal{S}_X(\Sigma)$ with pairwise disjoint $E_i \subset E$ and $M_\theta(f)\leq 1.$ The $(\theta, M)$ semi-variation of $\mu_\infty,$ for $E \in \Sigma$ by 
\begin{eqnarray*}
||\mu_\infty||_{\theta, M}(E)= \sup||\sum_{i=1}^{n}\mu_\infty(E_i)x_i||,
\end{eqnarray*}
 where the supremum is taken over all simple functions $f= \sum\limits_{i=1}^{n}x_ich(E_i)$ belonging to $\mathcal{S}_X(\Sigma)$ with pairwise disjoint $E_i \subset E$ and $M_\theta(f) \leq 1.$
 \begin{defn}\cite{Roy1} We call $\theta \in \Delta_2 $ if $\theta(2t) \leq k\theta(t),~k>0.$
 \end{defn}
 
\begin{thm}
$\mathcal{H}^\theta(M,X) $ is complete with respect to $||.||_{\theta, M}.$
\end{thm}
\begin{proof}
Suppose $(f_n)_{n=1}^{\infty}$ is a Cauchy sequence in $(\mathcal{H}^\theta(M,X),||.||_{\theta, M})$, then $||f_n -f_m||_{\theta, M} \leq \frac{1}{2}$ for $m>n.$ In general $n_k,~k \in \N$ such that $||f_n-f_m||_{\theta,M} < \frac{1}{2^k}$ for $m>n.$ Since $M$ is a $\sigma-$bounded, we can write $T= \bigcup\limits_{i=1}^{\infty}E_i,~E_i \cap E_j = \emptyset$ for $i \neq j$ and $M(E_i)< \infty$ for each $i=1,2,\dots,$ let $t>0$ be such that $M(E_i)\phi(t) \leq 1.$ If $g(x)=t $ on $E_i$ and $g(x)=0$ otherwise, then we have $M_\phi(g) \leq 1 $ and consequently $M(f_n -f_m) < \epsilon$ for $n,m \geq M_0$ so, $M(f_n -f_m) < \frac{\epsilon}{t}$  for $n,m \geq M_0.$ This gives $\{f_n\}$ is a Cauchy sequence in $\mathcal{S}_X(M).$ As $(f_n)$ is Cauchy sequence in $\mathcal{S}_X(M)$ then there exist a positive integer $M_1$ and a function $f_{M_1} \in \mathcal{S}_X(M)$ such that $M(f_n -f_{M_1}) < \frac{1}{2}$ for $ n \geq M_1.$ Similarly on $E_2,~M(f_n -f_{M_2})<\frac{1}{2^2}$  for $n \geq M_2$ and so on. The series $M(f_{M_1}) + M(f_{M_2}-f_{M_1}) +\dots$ is convergent. The fact $\mathcal{S}_X(M)$ is complete gives $(f_{M_1}) + (f_{M_2}-f_{M_1}) +\dots$ is converges a.e. to a function, say $f$ in $\mathcal{S}_X(M),$ so, $M(f_n -f)(g) < \epsilon$ whenever $n \geq M_0$ and $M_\phi(g) \leq 1.$ Thus $f_n -f \in \mathcal{H}^\theta(M,X)$ we have $ f \in \mathcal{H}^\theta(M,X).$ So $(\mathcal{H}^\theta(M,X), ||.||_{\theta, M})$ is complete.
\end{proof}
\begin{thm}\label{first}
Let $f_n \in \mathcal{H}^\theta(M,X).$ For $n=1,2,\dots$ then followings are equivalent:
\begin{enumerate}
\item $f \in \mathcal{H}^\theta(M,X)$ and $||f_n -f||_{\theta, M} \to 0 $ as $n \to \infty.$
\item $||f_n -f_m||_{\theta, M} \to 0 $ as $m,n \to \infty$ and hence there exists a subsequence $(f_{n_k}) $ of $(f_n) $ such that $f_{n_k} \to f ~M-$a.e.
\end{enumerate}
\end{thm}
\begin{proof}
Proofs are similar as Theorem 3.3 of \cite{Roy}.
\end{proof}
\begin{defn}
We denote $H(M,X) $ is the collection of all functions $f: T \to X$ which are $M-$measurable and for which $\theta(||f||) \in H(M,\R)$ where $H(M,\R) $ are Henstock-Kurzweil integrable function space.
\end{defn}
\begin{thm}\label{thm}
If $f \in \mathcal{H}^\theta(M,X) $ there exists a constant $N>0$ such that $Nf \in H(M,X).$ Moreover if $f \in \mathcal{H}^\theta(M,X)$ and $f(x) \neq 0$ then $$\sup\limits_{\mu_\infty \in M}H\int_{T}\theta\left(\frac{||f||}{||f||_{\theta,M}}\right)d|x^{\prime}\mu_\infty|\leq 1.$$
\end{thm}
\begin{proof}
Let $f \in \mathcal{H}^\theta(M,X)$ and $g \in \mathcal{S}(\Sigma),$ we need to prove for each $\mu_\infty \in M$,
\begin{eqnarray*}
H\int_{T}\frac{||fg||}{||f||_{\theta,M}}d|x^{\prime}\mu_\infty| \leq \max\{1, \mu_{\infty_\phi}(g)\}
\end{eqnarray*}
If $\alpha>1 $ and $t>0$ then $\theta(\alpha t) \geq \alpha \phi(t)$. So, let $f \in \mathcal{H}^\theta(M,X) $ and $g \in \mathcal{S}(\Sigma)$ then 
\begin{eqnarray}\label{2.4}
H\int_{T}||fg||d|x^{\prime}\mu_\infty| \leq ||f||_{\theta,M}~{\textit{~provided~}}~\mu_{\infty_\phi} \leq 1
\end{eqnarray}
If $ 1 < \mu_{\infty_\phi}(g) < \infty,$ then 
\begin{align*}
\phi(|g|) &= \phi\left(\mu_{\infty_\phi}\frac{|g|}{\mu_{\infty_\phi}(g)}\right)\\&\geq \mu_{\infty_\phi}(g)\phi\left(\frac{|g|}{\mu_{\infty_\phi}(g)}\right).
\end{align*}
That is $\frac{\phi(|g|)}{\mu_{\infty_\phi}(g)}\geq \phi\left(\frac{|g|}{\mu_{\infty_\phi}(g)}\right).$
So,
\begin{align*}
H\int_{T}\phi\left(\frac{|g|}{\mu_{\infty_\phi}(g)}\right)d|x^{\prime}\mu_\infty|&\leq H\int_{T}\frac{\phi(|g|)}{\mu_{\infty_\phi}(g)}d|x^{\prime}\mu_\infty|\\&\leq 1.
\end{align*}
So, $H\int_{T}\frac{||fg||}{\mu_{\infty_\phi}(g)}d|x^{\prime}\mu_\infty| \leq ||f||_{\theta, M}.$ This means
\begin{eqnarray}\label{2.5}
H\int_{T}||fg||d|x^{\prime}\mu_\infty| \leq ||f||_{\theta, M}\mu_{\infty_\phi}(g).
\end{eqnarray}
From (\ref{2.4}) and (\ref{2.5}) we get 
\begin{eqnarray*}
H\int_{T}||fg||d|x^{\prime}\mu_\infty| \leq ||f||_{\theta,M}\max \{1, \mu_{\infty_\phi(g)}\}\\
{\textit{~i.e}}~H\int_{T}\frac{||fg||}{||f||_{\theta,M}}d|x^{\prime}\mu_\infty| \leq \max \{1, \mu_{\infty_\phi}(g)\}.
\end{eqnarray*}
Next suppose that $ f \in \mathcal{H}^\theta(M,X)$ and $||f||_{\theta,M}>0$, from here we get two cases.\\
{\bf{Case 1}}: If $f$ is bounded and has support say $E,$ of finite $M-$measurable. Let $g = \frac{f}{||f||_{\theta,M}}$. Since $g$ is bounded on $E,~\phi(\chi(||g||))$ is bounded so, $\sup\limits_{\mu_\infty \in M}H\int_{T}\phi(\chi(||g||))d|x^{\prime}\mu_\infty|$ exists. \textcolor{green}{Now allowing the Young's inequality to an equality, we obtain}
\begin{align*}
&\quad \sup\limits_{\mu_\infty \in M}H\int_{T}\theta(||g||)d|x^{\prime}\mu_\infty|\\
&=\sup\limits_{\mu_\infty \in M}\left\{H\int_{T}\theta(||g||)d|x^{\prime}\mu_\infty|+ H\int_{T}\phi(\chi(||g||))d|x^{\prime}\mu_\infty|-H\int_{T}\phi(\chi(||g||))d|x^{\prime}\mu_\infty|\right\}\\&= \sup\limits_{\mu_\infty \in M}\left\{H\int_{T}||g\chi(||g||)d|x^{\prime}\mu_\infty|-H\int_{T}\phi(\chi(||g||))d|x^{\prime}\mu_\infty|\right\}.
\end{align*}
As $\chi(||g||)$ is positive bounded $M$-measurable function, there exists a non decreasing sequence of positive simple functions $\{g_n\}$ such that $g_n \to \chi(||g||)~M-$a.e. and $g_n \leq \chi(||g||)$ and $||gg_n|| \leq ||g \chi(||g||)||.$ So, \cite[Equation 7 of  page 353]{BD} $$\sup\limits_{\mu_\infty \in M}H\int_{T}||g\chi(||g||)||d|x^{\prime}\mu_\infty|\leq \sup\limits_{\mu_\infty \in M}\lim\limits_{n}H\int_{T}||gg_n||d|x^{\prime}\mu_\infty|.$$ Hence 
\begin{align*}
\sup\limits_{\mu_\infty \in M}H\int_{T}\theta(||g||)d|x^{\prime}\mu_\infty| &\leq \sup\limits_{\mu_\infty \in M}\max\{1-\mu_{\infty_\phi}(\chi||g||, 0)\\&\leq 1.
\end{align*}
{\bf{Case 2}}: Suppose $f \in \mathcal{H}^\theta(M,X)$ and $||f||_{\theta,M}>0.$ Since $M$ is $\sigma-$bounded, we can choose an increasing sequences of sets $\{E_n\}$ of finite $M-$measurable such that $T= \bigcup\limits_{n=1}^{\infty}E_n.$ Using (\ref{first}) for $M(||f_n|| \to M||f||) $ for each $n=1,2,\dots ,$ we have $$\sup\limits_{\mu_\infty \in M}H\int_{T}\theta\left(\frac{||f_n||}{||f_n||_{\theta,M}}\right)d|x^{\prime}\mu_\infty|\leq 1.$$
\end{proof}

\begin{prop}\label{del}
Suppose $\theta \in \Delta_2$ then $\mathcal{H}^\theta(M,X)= H(M,X)$ also if $f_n \in \mathcal{H}^\theta(M,X)$ with $\sup\limits_{\mu_\infty \in M}H\int_{T}\theta(||f_n||)d|x^{\prime}\mu_\infty| \to 0$ then $||f_n||_{\theta, M} \to 0.$
\end{prop}
\begin{proof}
Let $f: T \to X$ be $M-$measurable function and $g \in \mathcal{S}(\Sigma).$ Then by Young's inequality
\begin{eqnarray*}
H\int_{T}||fg||d|x^{\prime}\mu_\infty|\leq H\int_{T}\theta(||f||d|x^{\prime}\mu_\infty|.
\end{eqnarray*}
So, if $f \in H(M,X)$ then $f \in \mathcal{H}^\theta(M,X).$ Let $f \in \mathcal{H}^\theta(M,X)$ and $||f||_{\theta, M} \neq 0.$ Now from the Theorem \ref{thm}, $\frac{f}{||f||_{\theta,M}} \in H(M,X).$ Now as $f \in \mathcal{H}^\theta(M,X)$ implies $||f||_{\theta, M}<\infty$ that is, $||f||_{\theta,M} \leq 2^m;~m>0.$ Therefore
\begin{align*}
\theta(||f||)&= \theta\left(\frac{||f||}{||f||_{\theta,M}}\right)||f||_{\theta,M}\\&\leq 2^m \theta\left(\frac{||f||}{||f||_{\theta,M}}\right). 
\end{align*}
Thus, $\theta(||f||) \in H(M,X),$ so $ f \in H(M,X)).$ Hence $\mathcal{H}^\theta(M,X) = H(M,X).$
\end{proof}
\begin{thm}\label{simple}
$\mathcal{S}_X(\Sigma)$ is dense in $\mathcal{H}^\theta(M,X).$
\end{thm}
\begin{proof}
Choose arbitrarily $\epsilon>0$ and $f \in \mathcal{H}^\theta(M,X)$ that is, $ g \in H(M,X)$ with 
\begin{eqnarray*}
||g-f||_{\mathcal{H}^\theta(M,X)} \leq \frac{\epsilon}{N+1}.
\end{eqnarray*} 
Moreover, in correspondence with $\epsilon$ and $g,$ we find a $\Sigma-$simple function $s$, with
\begin{eqnarray*}
||s-g||_{H(M,X)} \leq \frac{\epsilon}{N+1}
\end{eqnarray*}
As $||.||_{\mathcal{H}^\theta(M,X)} \leq M||.||_{H(M,X)}$ and hence, we obtain
\begin{align*}
||s-f||_{\mathcal{H}^\theta(M,X)} &\leq ||s-f||_{\mathcal{H}^\theta(M,X)} + ||g-f||_{\mathcal{H}^\theta(M,X)}\\
&\leq N(||s-g||_{H(M,X)})+||g-f||_{\mathcal{H}^\theta(M,X)}\\&\leq \frac{N \epsilon}{N+1} +\frac{\epsilon}{N+1}\\&= \epsilon.
\end{align*}
So, $\mathcal{S}_X(\Sigma)$ is dense in $\mathcal{H}^\theta(M,X).$
\end{proof}
\begin{prop}\label{deb}
Suppose $\mu_\infty: \Sigma \to L(X,Y)$ is countable additive with $||\mu_{\infty_{X,Y}}||(T)< \infty,$ then $\mathcal{H}^\theta(M,X) \subset H(\mu_{\infty_{X,Y}}, X)$ and $ \mu_{\infty_{X,Y}}(f) \leq ||\mu_{\infty_{\theta,M}}||(T)||f||_{\theta, M}$ for $ f \in \mathcal{H}^\theta(M,X).$
\end{prop}
\begin{proof}
Let $f \in \mathcal{H}^\theta(M,X).$ Then by the Theorem \ref{simple}, there is a sequence $\{f_n\}$ of simple functions converging to $f$ in $\mathcal{H}^\theta(M,X).$ Let $y_1 \in Y^\prime $ be such that $|y_1| \leq 1.$ If $ f= \sum\limits_{i=1}^{n}x_i ch(E_i) \in \mathcal{S}_X(E),$ then $$H\int_{T}||f||d|\mu_{\infty_{i}}= \sum_{i=1}^{n}||x_i|||\mu_{\infty_{y_1}}|(E_i).$$ Suppose $\epsilon>0$ be arbitrary and for each $i,$ let $\{\mathcal{F}_{ij}\}$ be a finite family of disjoint sets in $\Sigma$ contained in $E_i$ such that 
\begin{eqnarray*}
|\mu_{\infty_{y_i}}|(E_i) < \sum\limits_{ij}||\mu_{\infty_{y_1}}(\mathcal{F}_{ij})||+ \frac{\epsilon}{n||x_i||}.
\end{eqnarray*}
Then 
\begin{eqnarray*}
H\int_{T}||f||d|\mu_{\infty_{y_1}}|<\sum\limits_{ij}||\mu_{\infty_{y_1}}(\mathcal{F}_{ij})||||x_i||+ \epsilon.
\end{eqnarray*}
For each $\mathcal{F}_{ij},$ we can choose an $x_{ij}$ with $||x_{ij}||=1 $ such that 
\begin{eqnarray*}
||\mu_{\infty_{y_1}}(\mathcal{F}_{ij})||<|\mu_{\infty_{y_1}}(\mathcal{F}_{ij}x_{ij})| + \frac{\epsilon}{k},
\end{eqnarray*}
where $k$ is a constant, by the Theorem \ref{thm}
\begin{align*}
H\int_{T}||f||d|\mu_{\infty_{y_1}}| &\leq \sum\limits_{ij}|\mu_{\infty_{y_1}}(\mathcal{F}_{ij})x_{ij}||x_{ij}|||+ \epsilon\\&= \sum\limits_{i,j}\left\vert\mu_{\infty_{y_1}}(\mathcal{F}_{ij})x_{ij} \frac{||x_i||}{||f||_{\theta,M}}\right\vert||f||_{\theta,M} + \epsilon \\&\leq |\mu_{\infty_{y_1}}|_{\theta,M}(T)||f||_{\theta,M}+ \epsilon.
\end{align*}
So,
\begin{eqnarray*}
H\int_{T}||f||d|\mu_{\infty_{y_1}}| \leq |\mu_{\infty_{y_1}}|_{\theta,M}(T)||f||_{\theta,M}.
\end{eqnarray*}
Therefore
\begin{eqnarray*}
\mu_{\infty_{X,Y}}(f) \leq ||\mu_{\infty}||_{\theta,M}(T)||f||_{\theta,M}.
\end{eqnarray*}
From above we have 
\begin{eqnarray*}
\mu_{\infty_{X,Y}}(f_n -f_m) \leq ||\mu_{\infty}||_{\theta,M}(T)||f_n -f_m||_{\theta,M}.
\end{eqnarray*}
This means $\{f_n\}$ is Cauchy sequence in $H(\mu_{\infty_{X,Y}},X).$ Thus $f \in H(\mu_{\infty_{X,Y}},X)$ as $f_n \to f$ is $\mu_{\infty_{X,Y}}$-a.e. and $\mu_{\infty_{X,Y}}(f_n -f) \to 0 $ as $ n \to \infty.$ So, $ \mu_{\infty_{X,Y}}(f) \leq ||\mu_{\infty_{\theta,M}}||(T)||f||_{\theta, M}$ for $ f \in \mathcal{H}^\theta(M,X)$ as this is true for each $n.$
\end{proof}
 We will now proof $(\mathcal{H}^\theta , ||.||_{\mathcal{H}^\theta})$ is a solid Banach lattice with weak order unit. Recalling a partially orderd Banach space $Z,$ which is also a vector lattice, is a Banach lattice if $||x|| \leq ||y|| $ for every $x,y \in Z$ with $|x| \leq |y|.$ A weak order unit of $Z$ is a positive element $ e \in Z$ such that if $x \in Z$ and $x\Lambda e=0,$ then $x=0.$ Let $Z$ be a Banach lattice and $ A \neq \emptyset \subset B \subset Z.$ We say that $A$ is solid in $B$ if for each $x,y $ with $x \in B,~y \in A$ and $|x| \leq |y|,$ it is $ x \in A.$
\begin{thm}
$(\mathcal{H}^\theta , ||.||_{\mathcal{H}^\theta})$ is a solid Banach lattice with weak order unit.
\end{thm}
\begin{proof}
It is clear that $(\mathcal{H}^\theta , ||.||_{\mathcal{H}^\theta})$ is a solid normed lattice with respect to the usual order. \\ By the Rybakov theorem (see also \cite{JJ} Theorem IX.2.2) there is $x_{1}^{\prime} \in X^\prime$ with $||x_{1}^{\prime}|| \leq 1 $ such that $\lambda= x_{1}^{\prime}\mu_\infty$ is a control measure of $\mu_\infty.$ If $f,g \in \mathcal{H}^\theta(|x^\prime \mu_\infty|),~|f| \leq |g|~\lambda$-a.e, $k \in \mathbb{N}$ and $ x^\prime \in X^\prime$ with $||x^\prime|| \leq 1$ then,
\begin{eqnarray*}
\inf\left\{k>0:~\theta\left(\frac{|f|}{k}\right)d|x^\prime \mu_\infty| \leq 1 \right\} \leq \inf\left\{k>0:~\theta\left(\frac{|g|}{k}\right)d|x^\prime \mu_\infty|\leq 1 \right\}.
\end{eqnarray*}
That is, $||f||_{\mathcal{H}^\theta} \leq ||g||_{\mathcal{H}^\theta}.$ So, $(\mathcal{H}^\theta(|x^\prime \mu_\infty|),||.||_{\mathcal{H}^\theta})$ is Banach lattice.
  Moreover, $\mathcal{H}^\theta(|x^\prime \mu_\infty|)$ is easily a weak order unit. 
\end{proof}
\begin{thm}
$(\mathcal{H}^\theta , ||.||_{\mathcal{H}^\theta})$ is a K\"othe function spaces.
\end{thm}
\section{Integral representations:}
We discuss the integral representations of operators from $\mathcal{H}^\theta(M,X)$ into $Y.$ We assume that $M$ is relatively weakly compact in $Ca(\Sigma),$ then we can find a control measure $\lambda: \Sigma \to \R^+$ for $M$ such that $\lambda \leq M$ and $M<<\lambda.$
\begin{prop}
Let $f: T \to C$ be a $\lambda-$measurable function such that $fg \in H(\lambda)$ for each $g \in \mathcal{H}^\theta(M),$ then $\sup\{H\int_{T}|fg|d\lambda:~g \in \mathcal{H}^\theta(M),~||g||_{\theta,M} \leq 1 \}< \infty.$
\end{prop}
\begin{proof}
If possible, let $\sup\{H\int_{T}|fg|d\lambda:~g \in \mathcal{H}^\theta(M),~||g||_{\theta,M} \leq 1 \}= \infty,$ then for each $n,$ there exists a sequence $\{g_n\} \subset \mathcal{H}^\theta(M)$ with $||g_n||_{\theta,M} \leq 1$ such that $H\int_{T}|fg_n|d\lambda \geq n.2^n.$ If $g_0(r)= \sum\limits_{n=1}^{\infty}\frac{1}{2^n}|g_n(r)|$ then 
\begin{align*}
||g_0||_{\theta,M} &\leq \frac{1}{2^n}||g_n||_{\theta,M}\\&\leq \sum_{n=1}^{\infty}\frac{1}{2^n}\\&=1.
\end{align*}
So, $g_0 \in \mathcal{H}^\theta(M)$ and hence $fg_0 \in H(\lambda).$ As $$ H\int_{T}|fg_0|d\lambda \geq \frac{1}{2^n}H\int_{T}|fg_n|d\lambda \geq n~{\textit{~for each}} ~n.$$ This gives $fg_0 $ does not belong in $ H(\lambda).$ Hence $$\sup\left\{H\int_{T}|fg|d\lambda: g \in \mathcal{H}^\theta(M),~||g||_{\theta,M} \leq 1 \right\}< \infty.$$
\end{proof}
We will now discuss the existence of countable additive measure for a continuous linear operator from $\mathcal{H}^\theta(M,X)$ into $Y.$
\begin{thm}
Suppose $\Delta: \mathcal{H}^\theta(M,X) \to Y$ is a continuous linear operator, then there exists a countable additive measure $\mu_\infty: \Sigma \to L(X,Y)$ such that 
\begin{eqnarray*}
\Delta f= H\int_{T}f d\mu_\infty~{\text{~for}}~f \in \mathcal{H}^\theta(M,X).
\end{eqnarray*}
\end{thm}
\begin{proof}
Suppose $\Delta: \mathcal{H}^\theta(M,X) \to Y$ is a continuous linear operator. Let us define $\mu_\infty: \Sigma \to L(X,Y)$ by $\mu_\infty(E)x= \Delta(x ch(E)).$ The linearity of $\mu_\infty(E),$ for each $E \in \Sigma$ is very obvious. Also,
\begin{align*}
||\mu_\infty(E)x||&=||\Delta(x ch(E))||\\&\leq||\Delta||||x ch(E)||_{\theta,M}\\&\leq||\Delta||||ch(E)||_{\theta,M}||x||.
\end{align*}
Therefore, $\Delta$ is continuous linear operator. Now we need to show $\mu_\infty $ is a countable additive measure.  Since
\begin{align*}
M_\theta(ch(E))=& \sup\limits_{\mu_\infty \in M}H\int_{T}\theta(ch(E))d\mu_\infty\\&= \theta(1)M(E).
\end{align*}
If $E_n \to \emptyset$ then $M_\theta(ch(E_n)) \to 0.$ Using Proposition \ref{del}, $||ch(E_n)||_{\theta,M} \to 0.$ Since $\mu_\infty$ is finitely additive, the above limit shows that $\mu_\infty$ is countable additive and $\Delta f= H\int_{T}f d\mu_\infty$ for $f \in \mathcal{S}_X(\Sigma).$ Now using Theorem \ref{simple}, $\Delta f= H\int_{T}f d\mu_\infty$ for $f \in \mathcal{H}^\theta(M,X).$
\end{proof}%\textcolor{blue}{
\begin{thm}\label{thm34}
If $\Delta: \mathcal{H}^\theta(M,X) \to Y$ defined by $\Delta f = H \int_{T}f d\mu_\infty$ for $ f \in \mathcal{H}^\theta(M,X)$ with $||\mu_\infty||_{\theta,M}(T)< \infty,$ then  $||\Delta|| \leq \mu_{\infty_{\theta,M}}(T) \leq 2||\Delta||.$
\end{thm}
We will now show the converse of  Theorem \ref{thm34} is hold for a countable additive measure.
\begin{thm}
If $\mu_\infty$ is a countable additive measure and  $\Delta: \mathcal{H}^\theta(M,X) \to Y$ defined by $\Delta f = H\int_{T}fd\mu_\infty$ for $ f \in \mathcal{H}^\theta(M,X)$ with $||\mu_\infty||_{\theta,M}(T)<\infty,$ then $\Delta$ is a continuous linear operator.
\end{thm}
\begin{proof}
Suppose $\mu_\infty: \Sigma \to L(X,Y)$ is a countable additive measure such that $||\mu_\infty||_{\theta,M}(T)< \infty,$ then the mapping $\Delta: \mathcal{H}^\theta(M,X) \to Y$ defined by $\Delta f= H\int_{T}fd\mu_\infty$ for all $f \in \mathcal{H}^\theta(M,X)$ is linear. Also,
\begin{align*}
||\Delta f||&=||H\int_{T}fd\mu_\infty||\\&=\sup\limits_{||y_1||\leq 1}\left\vert H\int_{T}d|y_1 \mu_\infty|\right\vert\\&\leq \sup\limits_{||y_1||\leq 1}H\int_{T}||f||d|y_1 \mu_\infty|\\&=\mu_{\infty_{X,Y}}(f)\\&\leq ||\mu_\infty||_{\theta,M}(T)||f||_{\theta,M}~{\text{using the Proposition \ref{deb}}}.
\end{align*}
Hence $\Delta: \mathcal{H}^\theta(M,X) \to Y$ is continuous.
\end{proof}
\section{Geometrical properties}
In this section, we will discuss about the uniformly convexity, reflexivity of $\mathcal{H}^\theta (M,X)$  and the Radon-Nikodym property of the Henstock-Orlicz spaces $\mathcal{H}^\theta(\mu_\infty).$ Before the discussion of uniformly convexity and reflexivity of $\mathcal{H}^\theta,$  we state about Modular spaces. The modular spaces can be defined as $B_{\theta,|x^{\prime}\mu_\infty|}(f)= H\int_{T}\theta(f)d|x^{\prime}\mu_\infty|.$
For complementary Young's function $\theta$ and $\phi,$ we define 
\begin{eqnarray}
\mu_{\infty_{\phi}}(g) = H\int_{T}\phi(|g|)d\mu_\infty,~g \in S(\Sigma)\\
M_\phi(g)= \sup\limits_{\mu_\infty \in M}\mu_{\infty_{\phi}}(g).
\end{eqnarray}
\begin{defn}
\cite{Lex} A Banach space $(X,||.||)$ is uniformly convex if for each $\epsilon>0$ there exists a $p(\epsilon) \in (0,1) $ such that $||x||=||y||=1,~||x-y||\geq \epsilon$ implies $||\frac{x+y}{2}|| \leq 1-p(\epsilon).$\\
 Similarly the modular $B_{\theta, |x^\prime  \mu_\infty|}$ is uniformly convex if for each $\epsilon>0$ there exists a $p(\epsilon) \in (0,1) $ such that $B_{\theta, |x^\prime  \mu_\infty|}(x)= B_{\theta, |x^\prime  \mu_\infty|}(y)=1,~B_{\theta, |x^\prime  \mu_\infty|}(x-y)\geq \epsilon$ implies $B_{\theta, |x^\prime  \mu_\infty|}(\frac{x+y}{2}) \leq 1-p(\epsilon).$
\end{defn}
\begin{prop}\label{pioneer1}
 Let $\theta$ satisfy $\Delta_2$ condition with $\mu_\infty$ is atom-less and $\mu_\infty(T) < \infty $ then $\mathcal{H}^\theta(M,X) $ is locally Uniformly convex if and only if pseudomodular $B_{\theta, |x^\prime \mu_\infty|}(M,X) $ is locally uniformly convex.
\end{prop}
\begin{proof}
Proof is similar to the  technique of  \cite[Lemma 3]{AK}.
\end{proof}
\begin{thm}
$\mathcal{H}^\theta(M,X) $ is uniformly convex if and only if the modular $B_{\theta, |x^\prime \mu_\infty|}(M,X) $ is uniformly convex.
\end{thm}
\begin{proof}
We can prove it by using  Proposition \ref{pioneer1} with similar technique of Lemma 1 \cite{AK1}.
\end{proof}
\begin{thm}
Assume $M$ is non atomic, then $\mathcal{H}^\theta(M,X) $ is reflexive.
\end{thm}
\begin{proof}
Since uniform convex implies reflexivity, so it  is obvious.
\end{proof}
Recalling a Banach space $X$ is said to have Radon-Nikodym property, shortly RNP if given a finite measure spaces $(T, \Sigma, \mu)$ and a vector measure $\mu_\infty:\Sigma \to X$ of finite variation and absolutely continuous with respect to $\mu_\infty$, then there exists a Bochner integrable function $g: \Omega \to X$ such that $\mu_\infty(E)= \int_{E}gd\mu$ for any $ E \in \Sigma.$ The fact that the  Bochner integral is McShane integral (also see \cite{Unknown}), every McShane integral is Henstock integral (see page 158 of \cite{Rh}). With this known fact we can define Radon-Nikodym property as follows for our setting:
\begin{defn}
  $\mathcal{H}^\theta(\Omega)$ is said to have Radon-Nikodym property, shortly RNP if given a finite measure spaces $(T, \Sigma, \mu)$ and a vector measure $\mu_\infty:\Sigma \to \mathcal{H}^\theta$ of finite variation and absolutely continuous with respect to $\mu_\infty$, then there exists a Henstock integrable function $g: T \to \mathcal{H}^\theta$ such that $\mu_\infty(E)= H\int_{E}gd\mu$ for any $ E \in \Sigma.$
  \end{defn}
  \begin{defn}
  We say that $ f \in X$ has an absolute norms if for every decreasing sequence $\{G_n\}$ of subset of $\Omega$ satisfying $\mu(G_n) \to 0$ we have $||f ch(G_n)|| \to 0.$
  \end{defn}
  \begin{thm}
  $(\mathcal{H}^\theta, H_\theta)$ has absolutely continuous norm.
  \end{thm}
  \begin{proof}
  Let $E_n = \{x \in \Omega:~n \leq \theta(x)< n+1\},~n \in \N.$ Then there exists a sequence of natural numbers $\{n_k\}_{k \in \N}$ such that $\mu_\infty(E_{n_k})>0.$ Let $c_k>0 $ be such that $H\int_{E_{n_k}}c_{k}^{\theta(x)} d|x^{\prime}\mu_\infty|=1,~k \in \N.$\\ We have $f(x)= \sum\limits_{k=1}^{\infty}c_k ch(E_{n_k})(x),~x \in \Omega~{\text{~and~}}~E_j = \bigcup\limits_{k=j}^{\infty}E_{n_k}.$ Then $E_n \to \phi.$\\
  Now 
  \begin{align*}
  ||f||&= \inf\left\{\lambda>0:~H\int_{E_{n_k}}\theta\left(\frac{f}{\lambda}\right)d|x^{\prime}\mu_\infty| \leq 1\right \}\\&\leq \inf\left\{\lambda>1:~\sum_{k=1}^{\infty}\left(\frac{1}{\lambda}\right)^{n_k} \leq 1 \right\}\\&\leq 2.
  \end{align*}
  So, $ f \in \mathcal{H}^\theta(\mu_\infty).$ Hence
  $$||f ch(E_n)||= \inf\left\{\lambda>0~:H\int_{E_{n_l}}\theta\left(\frac{c_l}{\lambda}\right)d|x^\prime\mu_\infty| \leq 1\right \}.$$ If $\mu_\infty(E_n) \to 0$ implies $||f ch(E_n)|| \to 0.$ So, $\mathcal{H}^\theta(\mu_\infty)$ has absolutely continuous norm.
  \end{proof}
   \begin{thm}
 Let $T= \bigcup\limits_{n=1}^{\infty}A_n $ be a union of measurable sets. If the subspace $\mathcal{H}_{n}^{\theta} =\{f \in \mathcal{H}^\theta:~f=0~{\textit{on}}~T \setminus A_n \}$ of $\mathcal{H}^\theta$ have RNP, then $\mathcal{H}^\theta$ have RNP.
  \end{thm}
  \begin{proof}
  Let $(T, \Sigma, \mu)$ be a finite measure space and let $\mu_\infty: \Sigma \to \mathcal{H}^\theta$ be a vector measure of finite variation which is absolutely continuous with respect to $\mu.$ We define projections $\mathcal{P}_n:\mathcal{H}^\theta \to \mathcal{H}_{n}^{\theta}$ by $\mcP_n(f)= f ch(A_n)$ and $\mu_{\infty_n}=\mcP_n(\mu_\infty)$ then each $\mu_{\infty_n}$ is an $\mathcal{H}_{n}^{\theta}-$valued vector measure of finite variation which is absolutely continuous with respect to $\mu.$ With the fact that each space $\mathcal{H}_{n}^{\theta}$ has the RNP, there exists a Henstock integrable function $h_n: T \to \mathcal{H}_{n}^{\theta}$ satisfying $\mu_{\infty_n}(E)= \int_{T}h_n d\mu$ for each $E \in \Sigma.$ Now for $E \in \Sigma$ and $n \in \mathbb{N}$ we have 
  \begin{eqnarray}\label{pioneer}
  H\int_{E}||\sum_{k=1}^{n}h_k||_{H_\theta} d\mu \leq |\mu_\infty|(E)
  \end{eqnarray}
  and so, for $\mu-$almost all $ \alpha \in \Omega,$ there exists $h(\alpha) \in \mathcal{H}^\theta$ such that $h(\alpha) = \sum_{k=1}^{\infty}h_k(\alpha).$ Using (\ref{pioneer}) and Fatou's lemma, $h$ is also Bochner integrable and so Henstock integrable. Lastly,
  \begin{align*}
  \mu_\infty(E) &=\lim\limits_{ n \to \infty}\sum_{k=1}^{n}\mu_{\infty_k}(E)\\&= \lim\limits_{ n \to \infty}\sum_{k=1}^{n}H\int_{E}h_n d\mu \\&= H\int_{E}g d\mu.
  \end{align*}
  Hence the proof.
  \end{proof}
  \section{Conclusion}
  In this article we discuss about Henstock-Orlicz space with vector measure. In Geometrical property we discuss about Uniform convexity, Reflexivity and finally about Radon Nikodym Property. Our one purpose of this article was to discuss about RNP with out $\Delta_2$ property, but we unable to proof our assumed result ``$\mathcal{H}^\theta(M,X)$ has RNP if and only if $X$ has RNP". In our research to proof the above result, we need Henstock differentiation. Interested Researcher can think about Henstock-differentiation with the technique of Bochner differentiation.
 
%\section{Weak Henstock-Orlicz space and vector measure}
%The weak Henstock-Orlicz space with respect to a vector measure $m$ and an N-function $\theta$ can be introduced as the following linear space
%\begin{eqnarray}
%\mathcal{H}_{w}^{\theta}(m) = \{ f \textit{~is bounded measurable} : ||f||_{H_w^\theta(m)} < \infty\} ~\\
%\textit{where~} ||f||_{H_w^\theta(m)}= \sup\{||f||_{\mathcal{H}^\theta}(|<m, x^{\ast}>|),~x^{\ast} \in B_{X^\ast}\}
%\end{eqnarray}
%for all $f$ is bounded measurable with compact support and coincides with the intersection of all scalar Henstock Orlicz space $\mathcal{H}^\theta(|<m, x^\ast>|) $ with $ x^\ast \in X^\ast$. The Henstock-Orlicz space with respect to the vector measure $m$ is defined as the closure of simple functions $\delta(\Sigma)$ in $\mathcal{H}_{w}^{\theta}(m)$ and will be denoted by $\mathcal{H}^{\theta}(m).$


\begin{thebibliography}{00}
\bibitem{MA} M.A-Algwaiz \textit{Theory of Distributions}, Pure and Applied Mathematics, Monograph, Marcel Dekker, Inc (1992). 

\bibitem {AL} A.  Alexiewicz,    Linear functionals on Denjoy-integrable functions,  Colloq. Math.  1(1948) 289--293.

\bibitem{CB} C. Bennett and R. Sharpley, {\it Interpolation of Operators}, Pure and Applied Mathematics Vol.
129, Academic Press, Boston, 1988.
\bibitem{BHK} A. Boccuto, B. Hazarika and H. Kalita, Kuelbs–Steadman Spaces for Banach
Space-Valued Measures, Mathematics, 2020, 8(6), 1005, 12 pages, DOI:10.3390/math8061005.
\bibitem{BD} J.K. Brooks and N. Dinculeanu, Lebesgue-type spaces for vector integration, linear
operators, weak completeness and weak compactness, J. Math. Anal. Appi. 54 (1976),
348--389
\bibitem{Th} T. K. Donaldson and N. S. Trudinger, Orlicz Sobolev space and Imbedding Theorems, J. Funct. Anal. 8(1971) 52--75. 

\bibitem{JD} J. Diestel and J. J. Uhl, Vector measures, Amer. Math. Soc., Providence, Rhode Island, (1977).
\bibitem{JJ} J.J. Duistermaat and J. A. C Kolk \textit{ Distribution: Theory and application}, Berlin Heidelberg, Newyork, Springer (2006).
\bibitem{etaltin} M. Et, Y. Altin, B. Choudhary and B.C. Tripathy, On some classes of sequences  defined by sequences of Orlicz functions, Math. Inequ. \& Appl. 9(2)(2006) 335-–342.

\bibitem{Unknown} D. L. Forkert, \textit{The Banach space-valued integrals of Riemann, Mcshane, Henstock-Kurzweil and Bochner}, Bachelor Thesis, Vienna University of Technology,
November 22, (2012).

%\bibitem{Tl} T. L. Gill, W. W. Zachary \textit{  Functional Analysis and The  Feynman Operator Calculus}, Springer International Publishing Switzerland, (2016). 
\bibitem{Ra}  R. A. Gordon, \textit{The Integrals of Lebesgue, Denjoy, Perron and Henstock}, Graduate Studies in Mathematics Vol 4, Amer. Math. Soc. (1994).


\bibitem{Rh}  R. Henstock \textit{The General Theory of Integration}, Oxford Mathematical Monographs, Clarendo Press, Oxford (1991). 

\bibitem{Rh1}R. Henstock, The equivalence of generalized forms of the Ward, variational, Denjoy–Stieltjes, and Perron–Stieltjes integrals Proc. London Math. Soc.  3(10) (1960)  281-–303.
\bibitem{HK} B. Hazarika and H. Kalita, Henstock Orlicz spaces and dense subspaces, Asian European J. Math. 14(7)(2021) 2150114 (17 pages)
DOI: 10.1142/S179355712150114X. % 2020, DOI: 10.1142/S179355712150114X.

\bibitem{KH} H. Kalita, B. Hazarika, Countable additivity of Henstock-Dunford Integrable
function and Orlicz Space, Anal. Math. Phys. 11(2)(2021), Article no 96, 1-13, https://doi.org/10.1007/s13324-021-00533-0
\bibitem{AK} A. Kaminska, The criteria for local Uniformly rotundity of Orlicz spaces, Studia Math. 79(3)(1984) 201--215.
\bibitem{AK1} A. Kaminska, On Uniform convexity of Orlicz spaces, Indagationes Mathematicae (Proceedings)  85(1)(1982) 27--36.
\bibitem{Ma} M.A. Krasonsel'skii, Ya. B. Rutickii, \textit{ Convex functions and Orlicz Spaces}, P. Noordfoff Ltd, Gronin-gen (1961). %MRI1281594.

 \bibitem{Kurzweil} J. Kurzweil, Generalized ordinary differential equations and continuous dependence on a parameter, Czechoslovak. Math. J. 7(82) (1957)  418--446.
%\bibitem{Lp} L. Paxton, A Sequential approaches to the Henstock Integral, arxiv:1609.05454v1 (2016).
\bibitem{Lex}W.A.J. Luxemburg, Banach function spaces, Thesis, Delft (1955).
\bibitem{Mm} M.M. Rao and Z.D. Ren, {\it Theory of Orlicz Spaces}, Vol 146 of Pure and Applied Mathematics, Marcel. Dekker, Inc (1991).
\bibitem{Roy} S. K. Roy and N. D. Chakraborty, Orlicz spaces for a family of measures, Analysis Mathematica, 12 (1986), 229--235.
\bibitem{Roy1}S. K. Roy and N. D. Chakraborty,On Orlicz Spaces for a Family of Measures, J. Math. Anal. Appl. 145(1990) 485--503. 
\bibitem{BP} Ben-Zion A. Rubshtein, G. Ya. Grabarnik, M.A. Muratov and Y. S. Pashkova \textit{Foundations of Symmetric Spaces of Measurable Functions}, Developments in Mathematics, Springer, (2016).
\bibitem{nsrnns} N. Subramanian, R. Nallaswamy and N. Saivaraju, Characterization of Entire Sequences via Double Orlicz Space, Inter. J. Math.  Math.  Sci.  2007, Article ID 59681, 10 pages.
%\bibitem{S} S. Schwabik,  Y. Guoju, \textit{Topics in Banach Spaces Integration, Series in Real Analysis}, Vol 10. World scientific, Singapore (2005). 
\bibitem{Cs} C. Swartz, \textit{Introduction to Gauge Integral}, World Scientific Pub. Co. (2001). 

\bibitem{Bs}  B. S. Thomson, \textit{ Theory of Integral}, Classical Real Analysis. Com. (2008).

\bibitem{Tv}  T. V. Thung,  Some collections of functions Dense in an Orlicz Space, Acta Math. Vietnamica.  25(2)(2000)  195--208.


\bibitem{bcrg} B.C. Tripathy and R. Goswami,  Vector Valued Multiple Sequence Spaces Defined by Orlicz Function, Bol. Soc. Paran. Mat.  33(1)(2015) 69-–81.
%\bibitem{Bs}  B. S. Thomson, \textit{ Theory of Integral}, Classical Real Analysis. Com (2008).

\bibitem{lee} L. T. Yeong, {\it Henstock–Kurzweil Integration on Euclidean spaces},  Series in Real Analysis, Vol. 12, World Scientific Publishing (2011).

\end{thebibliography}
\end{document}